\documentclass{article}
\usepackage{graphicx} 
\usepackage{mathalfa}
\usepackage{mathrsfs}
\usepackage{amsmath,amsthm,enumerate}
\usepackage{indentfirst,graphicx,epsfig}

\input{epsf}

\usepackage{epstopdf}
\usepackage{tikz,pgf}
\usepackage{subfig}
\usepackage{caption}
\usepackage{array}
\usepackage{makecell}
\usepackage{tabularx}
\usepackage{amssymb}

\usetikzlibrary{shapes.geometric} 

\usetikzlibrary{positioning}
\usetikzlibrary{decorations.markings}
\usetikzlibrary{decorations.pathmorphing}
\usetikzlibrary{patterns}
\tikzset{->-/.style={decoration={
  markings,
  mark=at position .5 with {\arrow[scale=0.8]{>}}},postaction={decorate}}}
\tikzset{snake it/.style={decorate, decoration={snake, amplitude=.4mm, segment length=2mm}}}

\newtheorem{thm}{Theorem}[section]
\newtheorem{df}[thm]{Definition}

\newtheorem{cor}[thm]{Corollary}

\newtheorem{lem}[thm]{Lemma}
\newtheorem{pro}[thm]{Proposition}
\newtheorem{question}[thm]{Question}

\title{\textbf{Cooperative colorings of hypergraphs}}
\author{Xuqing Bai$^1$\footnote{Supported by the Fundamental Research Funds for the Central Universities (No.\,ZYTS24069).} \quad\quad Bi Li$^1$\thanks{Corresponding author.} \quad\quad Weichan Liu$^2$ \quad\quad Xin Zhang$^1$\footnote{Supported by the Natural Science Basic Research Program of Shaanxi (No.\,2023-JC-YB-001)
and the Fundamental Research Funds for the Central Universities (No.\,ZYTS24076).}\\
{\small 1. School of Mathematics and Statistics, Xidian University, Xi'an, 710071, China}\\
{\small $\{$baixuqing, libi, xzhang$\}$@xidian.edu.cn}\\
{\small 2. School of Mathematics, Shandong University, Ji'nan, 250100, China}\\
   {\small wcliu@sdu.edu.cn}\\ 
}
\date{}

\begin{document}

\maketitle
\begin{abstract}\baselineskip 0.50cm
    Given a class $\mathcal{H}$ of $m$ hypergraphs ${H}_1, {H}_2, \ldots, {H}_m$ with the same vertex set $V$, a cooperative coloring of them is a partition $\{I_1, I_2, \ldots, I_m\}$ of $V$ in such a way that each $I_i$ is an independent set in ${H}_i$ for $1\leq i\leq m$. The cooperative chromatic number of a class $\mathcal{H}$ is the smallest number of hypergraphs from $\mathcal{H}$ that always possess a cooperative coloring. For the classes of $k$-uniform tight cycles, $k$-uniform loose cycles, $k$-uniform tight paths, and $k$-uniform loose paths, we find that their cooperative chromatic numbers are all exactly two utilizing a new proved set system partition theorem, which also has its independent interests and offers a broader perspective. For the class of $k$-partite $k$-uniform hypergraphs with sufficient large maximum degree $d$, 
    we prove that its cooperative chromatic number has lower bound $\Omega(\log_k d)$ and upper bound 
    $\text{O}\left(\frac{d}{\ln d}\right)^{\frac{1}{k-1}}$.
\end{abstract}

\noindent\textbf{Keywords:} cooperative coloring; chromatic number; set partition; multipartite hypergraphs.

\noindent\textbf{AMS subject classification 2020:} 05C15, 05C65, 60C05.

\baselineskip 0.50cm

\section{Introduction}
A {\it hypergraph} is a pair $(V,E)$ where $V$ is a set whose elements are called vertices, and $E$ is a
 family of subsets of $V$ called edges. A hypergraph is {\it $k$-uniform} if every edge contains exactly $k$ vertices.  
A \textit{proper coloring} of a hypergraph involves assigning colors to its vertices in such a way that no edge in the hypergraph becomes monochromatic; meaning, no edge contains vertices of the same color.
A vertex subset $I$ of a hypergraph $H = (V, E)$ is defined to be an \textit{independent set} if it does not fully contain any edge $e\in E$. In simpler terms, an independent set in a hypergraph is a collection of vertices such that no edge in the hypergraph comprises solely of vertices from that set.  

A \textit{cooperative coloring} in a family $G_1, G_2, \ldots, G_m$ (not necessarily distinct) of graphs that all share the same vertex set $V$ is defined as a process of selecting one independent set $I_i$ from each graph $G_i$ for every $i \in [m]: = \{1, 2, \ldots, m\}$, in such a way that the union of all these independent sets covers the entire vertex set $V$, i.e.,    
$\bigcup_{i=1}^{m} I_i = V$.
The notion of cooperative coloring was initially introduced by Aharoni \textit{et al}.\,\cite{AHHS}  
and has since garnered significant attention and extensive research, with numerous studies conducted on the topic
\cite{ABCHJ,AHHS,BLXZ,bartnicki2023cooperativecoloringmatroids,B,BM}.
In this paper, we broaden the scope of the notion of cooperative coloring by generalizing it to hypergraphs.

\begin{df}
Given a family ${H}_1, {H}_2, \ldots, {H}_m$ (not necessarily distinct) of hypergraphs with the same vertex set $V$, a cooperative coloring of them is a partition $\{I_1, I_2, \ldots, I_m\}$ of $V$ such that $I_i$ is an independent set in ${H}_i$ for each $i\in [m]$.  
\end{df}

The notion of cooperative coloring in a family of hypergraphs can be applied in various domains. The key idea is that by partitioning the vertex set in a way that respects the independence structure of each hypergraph, one can identify useful subsets of vertices that can be treated separately or together in a manner that satisfies the constraints or requirements imposed by the hypergraphs.
For example, in data mining and machine learning, a hypergraph can represent complex relationships among data points. A cooperative coloring of a family of hypergraphs, each capturing a different aspect of the data, can be used to partition the data into clusters that are coherent within each hypergraph's representation but also provide a balanced view across all representations.
In cooperative game theory, where players form coalitions to achieve shared goals, a hypergraph can represent the possible coalitions among players. A cooperative coloring of a family of hypergraphs, each capturing different constraints or preferences among players, can be used to find stable partitions of players into coalitions that satisfy all relevant constraints and preferences.

Certainly, the concept of cooperative coloring in a family of hypergraphs represents an extension of the notion of proper coloring within a single hypergraph. In the special case where all $H_i$ ($i \in [k]$) within the family are identical, being the same hypergraph $H$, a cooperative coloring of this family simply reduces to a proper coloring of $H$ using $k$ distinct colors. This means that each of the $k$ color classes, or independent sets, forms a partition of the vertices of $H$ such that no hyperedge within $H$ consists of vertices of the same color class.

The renowned Brooks' theorem establishes a fundamental fact that any graph having a maximum degree of $d$ admits a proper  coloring using at most $d+1$ colors.  
This result has been extended to hypergraph by Kostochka, Stiebitz, and Wirth \cite{zbMATH01019041}, demonstrating that 
$d+1$ colors are sufficient to properly color any hypergraph with maximum degree of at most $d$.
It is natural to explore degree conditions that ensure the existence of cooperative colorings, akin to how the maximum degree of a (hyper)graph serves to bound its cooperative chromatic number (we formalize it later). In this framework, Aharoni \textit{et al}.\,\cite{AHHS} defined the parameter $m(d)$ as the minimum number of graphs that share the same vertex set and have a maximum degree of at most $d$, all of which inherently possess a cooperative coloring.
By utilizing the fundamental theorems on independent transversals established by Haxell \cite[Theorem 2]{H}, it can be shown that any collection of $2d$ graphs, where each graph has a maximum degree of $d$, can always have a cooperative coloring, meaning $m(d)\leq 2d$. Analogously, we have the following proposition for hypergraphs.

\begin{pro}\label{proposition}
Every $2d$ hypergraphs that share the same vertex set and have a maximum degree of at most $d$ inherently possesses a cooperative coloring.
\end{pro}

In actuality, let $\mathcal{H}$ represent a collection of $2d$ hypergraphs, specifically $\{H_1, H_2, \ldots, H_{2d}\}$, where each hypergraph shares the same vertex set $V$ and has a maximum degree of at most $d$. Subsequently, we  construct a family of graphs, denoted as $G_1, G_2, \ldots, G_{2d}$, with the identical vertex set $V$ and a maximum degree also bounded above by $d$. This construction is based on the given family of hypergraphs $\mathcal{H}$, and the process is outlined as follows.

Given any $H_i$ with $i \in [2d]$, let $\ell_i = |E(H_i)|$. For each hyperedge $e_j\in E(H_i)$, where $j \in [\ell_i]$, we proceed to construct a corresponding edge $e_j'$ in the graph $G_i$. This construction involves selecting two distinct vertices from the set of vertices that constitute the hyperedge $e_j$. By repeating this process for all hyperedges in $H_i$, we obtain a simple graph $G_i$ whose edge set $E(G_i)$ is the union of all the constructed edges $e_j'$ for $j \in [\ell_i]$, and whose vertex set $V(G_i)$ remains the same as that of the hypergraphs, denoted as $V$. Note that in $G_i$ there may exist some isolated vertices, but we are not concerned about this. It is evident that each $G_i$ has maximum degree of at most $d$.

Since $m(d)\leq 2d$, the family $G_1, G_2, \dots, G_{2d}$ of graphs possesses a cooperative coloring.
This means there is a partition $\{I_1,I_2,\ldots,I_{2d}\}$ of $V$ such that each $I_i$ is an independent set of $G_i$. Now, no edge of $H_i$ would be fully contained in  $I_i$, as otherwise the two vertices of some edge $e_j'$ of $G_i$ would be contained in $I_i$, contradicting the fact that $I_i$ is independent in $G_i$. Therefore, 
$\{I_1,I_2,\ldots,I_{2d}\}$ becomes a partition of $V$ such that each $I_i$ is an independent set of $H_i$.

In the realm of graph coloring problems, researchers often focus on exploring the coloring properties within a distinct class of graphs. Consequently, we are able to restrict the cooperative coloring to specific classes of hypergraphs. Formally, we define the following:

\begin{df}
For a given class of hypergraphs $\mathcal{H}$, the cooperative chromatic number of $\mathcal{H}$, denoted as 
$\chi_{{\rm co}}(\mathcal{H})$, is the least number of hypergraphs from $\mathcal{H}$ that always have a cooperative coloring. 
\end{df}
\noindent \textbf{Remark:} If each hypergraph form $\mathcal{H}$ has a maximum degree of at most $d$, then we 
adopt $m_\mathcal{H}(d)$ as an alternative representation to $\chi_{{\rm co}}(\mathcal{H})$, just following an analogous notion of Aharoni \textit{et al}.\,\cite{ABCHJ}.

For (graph) cycles and paths, Aharoni \textit{et al}.\,\cite{AHHS} showed that three cycles on the same vertex set possess a cooperative coloring, which by extension, implies that three paths sharing the identical vertex set also possess a cooperative coloring. It is noteworthy to highlight that their proof does not follow a combinatorial approach.
For a given class of bipartite graphs $\mathcal{B}$, Aharoni \textit{et al}.\,\cite{ABCHJ} proved
\[\Omega(\log_2 d) \leq m_{\mathcal{B}}(d) \leq \text{O}\left(\frac{d}{\ln d}\right).\]

In this paper, we extend those results to the corresponding hypergraphs, roughly say, hypercycles, hyperpaths, and multipartite hypergraphs. Denote $\mathbb{Z}_n$ as the set of integers $\{0, 1, 2, \ldots, n-1\}$   
with addition and multiplication defined in such a way that the result is always   
taken modulo $n$.

\begin{df}\label{df:tcycle}
A $k$-uniform tight cycle is a hypergraph with a sequence of $n$ vertices $v_0, v_1, \dots, v_{n-1}$ such that for each $i\in \mathbb{Z}_n$, the set $\{v_i, v_{i+1}, \dots, v_{i+(k-1)}\}$ (with indices taken modulo $n$) forms an edge.
A $k$-uniform loose cycle is a hypergraph with a sequence of edges $e_0, e_1, \dots, e_{m-1}$ such that for each $i \in \mathbb{Z}_m$, $e_i \cap e_{i+1}=v_i$, where $e_{m}=e_0$, and all $v_i$-s are distinct (non-consecutive edges do not share any vertices). A $k$-uniform tight or loose path can be defined using similar logical constructs. 
\end{df}

\begin{df}
Given an integer $k\ge 3$, a $k$-uniform hypergraph is $k$-partite if its vertex set $V$ can be partitioned into $k$ disjoint subsets $V_1, V_2, \ldots, V_k$
  such that $V=V_1\cup V_2\ldots \cup V_k$, and every edge intersects each subset $V_i$ in exactly one vertex for all $i\in [k]$. Furthermore, this hypergraph is a complete $k$-partite $k$-uniform hypergraph if, for any selection of vertices $ v_i\in V_i$ for $i\in [k]$, the set $\{v_1,v_2,\ldots, v_k\}$ forms an edge.
\end{df}

The rest of the paper is organized as follows.
In Section \ref{sec:2}, we show that
$\chi_{{\rm co}}(\mathcal{G})=2$ if $\mathcal{G}$ represents 
the class of $k$-uniform tight cycles, $k$-uniform loose cycles, $k$-uniform tight paths, or $k$-uniform loose paths, where $k\geq 3$.
Instead of proving them separately, we prove a conclusion about set system partition problem, which can be used to deduce any of these results. Moreover, this conclusion itself also has independent interest.
In Section \ref{sec:3}, we move our attention to the the class $\mathcal{B}_k$ of $k$-uniform $k$-partite hypergraphs, proving that
\[\Omega(\log_k d)\le m_{\mathcal{B}_k}(d)\le \text{O}\left(\frac{d}{\ln d}\right)^{\frac{1}{k-1}}.\]
In Section \ref{sec:4}, we wrap up the paper by highlighting some interesting open questions that are worth pursuing further.

\section{Hypercycles and hyperpaths} \label{sec:2}

The following is the main result of this section.

\begin{thm}\label{thm:2partition}
Let $a_0,a_1,\ldots,a_{n-1}$ be a permutation of $\mathbb{Z}_n$.
Define $N_i=\{ i,i+1,\dots, i+k_i-1\}$ and $A_i=\{a_i, a_{i+1}, \dots, a_{i+j_i-1}\}$
for each $i\in \mathbb{Z}_n$, where all indices are taken modulo $n$ to ensure they remain within $\mathbb{Z}_n$. Furthermore, suppose that the size of each set $N_i$ and $A_i$ is at least 3, with the exception that $|N_0|$ and $|A_0|$ are at least 2.
Then, there exists a partition $\{B,R\}$ of $\mathbb{Z}_n$ such that no set $N_i$ is fully contained within $B$ and no set $A_i$ is fully contained within $R$ for any $i\in\mathbb{Z}_n$.
\end{thm}

\begin{proof}
Construct a graph $G$ defined by the vertex set $V:=\mathbb{Z}_n$ and an edge set $E$ that comprises all the edges belonging to two cycles: 
the first cycle $C_1$ traversing vertices in the order $0,1,\ldots,n-1$ in a continuous loop, and the second cycle $C_2$ following the sequence of vertices 
$a_0,a_1,\ldots a_{n-1}$ in a similar continuous fashion. Observe that the graph $G$ can potentially be a multigraph.
For the purpose of clarity and simplified description, color the edges of $C_1$ 
blue, and the edges of $C_2$ red. Consequently, every vertex in $G$ is incident with exactly two blue edges (from $C_1$) and two red edges (from $C_2$).
To achieve the desired partition of the vertex set $V$, we utilize a proper  coloring of certain subgraphs of $G$. This approach necessitates the consideration of two distinct cases.

If $n=2\ell$ for some $\ell$, then we proceed by removing a maximum matching from $C_1$, comprising $\ell$ blue edges, none of which is the edge $\{0,1\}$. Similarly, from $C_2$, we delete a maximum matching consisting of $\ell$ red edges, ensuring that the edge $\{a_0,a_1\}$ is not included in this matching.
Denote the resulting graph by $G'$. Then, in $G'$, every vertex is incident with precisely one blue edge and one red edge. Furthermore, the two edges $\{0,1\}$ and
$\{a_0,a_1\}$ remain as part of the edge set of $G'$, i.e., $\{0,1\}, \{a_0,a_1\}\in E(G')$.
Consequently, every connected component of $G'$ is an even cycle, and therefore $G'$ has a proper 2-coloring, which in turn partitions $V$ into two disjoint subsets, denoted by $B$ and $R$.
The partition $\{B,R\}$ of $V$ ensures that no two adjacent vertices in $G'$ are simultaneously contained within either $B$ or $R$.
Since $\{0,1\}\in E(G')$, $N_0 \not\subseteq B$.
For any $i\in \mathbb{Z}_n\setminus \{0\}$, either $\{i,i+1\}$ or $\{i+1,i+2\}$ is an edge of $G'$. This implies
$\{i,i+1,i+2\}\not\subseteq B$, and thus $N_i\not\subseteq B$. By symmetry, we can analogously argue that $A_i\not\subseteq R$ for any $i\in \mathbb{Z}_n$.

If $n=2\ell+1$, then we divide the proof into two subcases according to whether $D:=\{2,4,\dots, 2\ell\}\cap \{a_2, a_4, \dots, a_{2\ell}\}$ is empty or not.

\begin{itemize}
        \item If $D\not=\emptyset$, then choose $v\in D$. Upon removing $v$, the cycle $C_1$ (resp.\,$C_2$) transforms into a path denoted as $P_1$ (resp.\,$P_2$), which comprises $2\ell-1$ edges all colored blue (resp.\,red).
        The maximum matching in $P_1$ (resp.\,$P_2$) is unique and comprises $\ell$ blue (resp.\,red) edges, which we denote by $M_1$ (resp.\,$M_2$).
        Observe that $\{0,1\}\in M_1$ and $\{a_0,a_1\}\in M_2$.
Let $G''$ be the subgraph of $G$ induced by the vertex set $V\setminus \{v\}$ and the edge set $M_1\cup M_2$. 
Since $M_1$ and $M_2$ are two perfect matchings of $G''$, $G''$ can be expressed as a union of even cycles, which implies that $G''$ is 2-colorable.
Assume $v=a_j$ for some $2\le j\le 2l$ and $u=a_{j-1}$, i.e., $u$ is a vertex in $V$ such that $\{u,v\}$ is a red edge (i.e., an edge of $C_2$).
Then $\{a_{j+1},a_{j+2}\}\in M_2$. 
As in the previous case, we obtain a partition $\{B',R'\}$ of $V\setminus \{v\}$ with the property that $N_i\not\subseteq B'$ and $A_i\not\subseteq R'$ for any $i\in \mathbb{Z}_n$, and furthermore, $u\in B'$. Let $B=B'$ and $R=R'\cup \{v\}$. Subsequently, we demonstrate that $\{B, R\}$ constitutes the desired partition of the graph $G$ as outlined below.
It suffices to show that for any set $A\in \{A_2, A_3, \dots, A_n\}$ that includes $v$,  it holds that $A\not\subseteq R$.
If $u\in A$, then $A\not\subseteq R$ as $u\in B$.
If $u\not\in A$, then $\{a_j,a_{j+1},a_{j+2}\}\subseteq A$.
Since $\{a_{j+1},a_{j+2}\}$ is an edge of $M_2$, $a_{j+1}$ and $a_{j+2}$ cannot belong to $B'$ or $R'$ at the same time.
It follows $A\not\subseteq R$.
\item If $D=\emptyset$, then $\{a_2, a_4, \dots, a_{2\ell}\} \subseteq \{0,1,3, 5, \dots, 2\ell-1\}$.
        
    Suppose first that there exists $k\in \{2, 4, \dots, 2\ell-2\}\setminus \{a_0,a_1\}$ such that $k+1\in \{a_2, a_4, \dots, a_{2l}\}\setminus \{0,1\}$.
        Assume $k=a_x$ and $k+1=a_y$. 
        We assume $x<y$, and it is reasonable to infer that the case where $x>y$ can be demonstrated through an analogous approach.
        Since $k$ and $y$ are both even, we can always choose $M_3$ and $M_4$ be maximum matchings of
        $C_1\setminus \{k,k+1\}$ and $C_2\setminus \{a_x,a_y\}$ such that $\{0,1\}, \{k-2,k-1\}\in  M_3$ and $\{a_0,a_1\}, \{a_{y+1},a_{y+2}\}\in M_4$.
        Let $G'''$ be the subgraph of $G$ with edge set $M_3\cup M_4$ and vertex set $V\setminus \{k,k+1\}$.
        Subsequently, each vertex in $G'''$ is adjacent to at most one blue edge and at most one red edge. Analogous to the first case, a partition $\{B'',R''\}$ of $V\setminus \{k,k+1\}$ emerges neither $N_i$ is fully contained within $B''$ nor $A_i$ within $R''$ for any $i\in \mathbb{Z}_n$.
        We set $a_{y-1}\in B''$ if $y-1\neq x$.
        Let $R=R''\cup \{k+1\}$ and $B=B''\cup \{k\}$. We now demonstrate that the pair $\{B,R\}$ constitutes the desired partition of $G$, as follows.

        Since $\{0,1\}\in M_3$ and $\{a_0,a_1\}\in M_4$, it is impossible for 0 and 1 to both belong to $B$ or $R$ simultaneously, and similarly, it is impossible for $a_0$ and $a_1$ to both belong to $B$ or $R$ simultaneously. This implies $N_0\not\subseteq B$ and $A_0\not\subseteq R$.
        Let $A\in \{A_2, A_3, \dots, A_n\}$. If $a_y\not\in A$, then it is clear that $A\not\subseteq R$.
        So we assume $a_y\in A$. Since $a_{y-1}\in B$ (if $y-1=x$, then $a_x\in B$ definitely) and 
        $\{a_{y+1},a_{y+2}\}\in M_4$, we conclude that $\{a_{y-2},a_{y-1},a_y\}, \{a_{y-1},a_{y},a_{y+1}\}, \{a_{y},a_{y+1},a_{y+2}\} \not\subseteq R$. This implies $A\not\subseteq R$.
        Let $N\in \{N_2, N_3, \dots, N_n\}$. If $k\not\in N$, then it is clear that $N\not\subseteq B$.
        If $k\in N$, then $\{k-2,k-1,k\}, \{k-1,k,k+1\}, \{k+1,k,k+2\} \not\subseteq B$ because $\{k-2,k-1\}\in M_3$ and $k+1\in R$.
        This concludes $N\not\subseteq B$. 

        On the other hand, suppose that for each $k\in \{2, 4, \dots, 2\ell-2\}\setminus \{a_0,a_1\}$, we have $k+1\not\in \{a_2, a_4, \dots, a_{2\ell}\}\setminus \{0,1\}$.
        If $\ell=1$ (i.e., $n=3$), then we can trivially find the desired partition $\{B,R\}$, so we assume $\ell\geq 2$ below.
        
        If $0$ or $1$ is missing in $\{a_2, a_4, \dots, a_{2\ell}\}$, then $\{3,5,\ldots, 2\ell-1\}\subseteq \{a_2, a_4, \dots, a_{2l}\}$. 
        If $\ell\geq 4$, then $\{2, 4, \dots, 2\ell-2\}\setminus \{a_0,a_1\}$ is nonempty and contains a $k$ such that $k+1\in \{3,5,\ldots, 2\ell-1\}$, contradicting our assumption.
        Hence $\ell\leq 3$.

        \begin{itemize}
            \item If $\ell=2$, then $2\in \{a_0,a_1\}$. 
        Given $B=\{0,2,3\}$ and $R=\{1,4\}$, we derive a desired partition.
      
       \item If $\ell=3$, then $\{a_0,a_1\}=\{2,4\}$. Let $\{\alpha,\beta\}=\{0,1\}$ such that $\beta\in \{a_2, a_4, a_6\}$. Setting $B=\{\alpha,2,4,6\}$ and $R=\{\beta,3,5\}$, we arrive at a desired partition. 
        \end{itemize}
        
        If some odd $t\geq 3$ is missing in $\{a_2, a_4, \dots, a_{2\ell}\}$, then $\{a_2, a_4, \dots, a_{2\ell}\} = \{0,1,3, 5, \dots, 2\ell-1\}\setminus \{t\}$.
        If $\ell\geq 5$, then $\{2, 4, \dots, 2\ell-2\}\setminus \{a_0,a_1\}$ has at least two elements and contains a $k$ such that $k+1\in \{3,5,\ldots, 2\ell-1\}\setminus \{t\}$, contradicting our assumption.
        Hence $\ell\leq 4$. 
        \begin{itemize}
            \item
        If $\ell=2$, then $\{a_2,a_4\}=\{0,1\}$. By defining  $B=\{0,2,3\}$ and $R=\{1,4\}$, we achieve a desired partition.
        
        \item If $\ell=3$, then $\{a_2,a_4,a_6\}=\{0,1,3\}$ and $2\in \{a_0,a_1\}$, or $\{a_2,a_4,a_6\}=\{0,1,5\}$ and $4\in \{a_0,a_1\}$. 
        In the former case, setting $B=\{2,4,5,6\}\setminus\{\gamma\}$ and $R=\{0,1,3,\gamma\}$, where 
        $\gamma\in \{a_0,a_1\}\setminus \{2\}$ (thus $\gamma\in \{4,5,6\}$), 
        we achieve a desired partition; in the latter case, $B=\{2,3,6\}$ and $R=\{0,1,4,5\}$ give a desired partition.
        
      \item If $\ell=4$ and $t=3$ (resp.\,$5$), then $\{a_2, a_4, a_6, a_{8}\} = \{0,1,5,7\}$ and $\{a_0,a_1\}=\{4,6\}$ (resp.\,$\{a_2, a_4, a_6, a_{8}\} = \{0,1,3,7\}$ and $\{a_0,a_1\}=\{2,6\}$). Setting $B=\{2,3,6,8\}$ and $R=\{0,1,4,5,7\}$ (resp. $B=\{2,4,5,8\}$ and $R=\{0,1,3,6,7\}$ ), we obtain a desired partition.
      \item
      If $\ell=4$ and $t=7$, then $\{a_2, a_4, a_6, a_{8}\} = \{0,1,3,5\}$.
      It follows $\{a_0,a_1\}=\{2,4\}$ and $\{a_3,a_5,a_7\}=\{6,7,8\}$.
      If $3\in \{a_4,a_6\}$, then $B=\{2,3,a_3,a_7\}$ and $R=\{0,1,4,5,a_5\}$ give a desired partition;
      If $3=a_2$ (resp.\,$3=a_8$), then $B=\{2,3,a_5,a_7\}$ and $R=\{0,1,4,5,a_3\}$ (resp.\,$B=\{2,3,a_3,a_5\}$ and $R=\{0,1,4,5,a_7\}$) give a desired partition.
      \end{itemize}
      \end{itemize}
This completes the proof.
\end{proof}

Denote $\mathcal{C}^{{\rm tyt}}_k$, $\mathcal{C}^{{\rm lse}}_k$, $\mathcal{P}^{{\rm tyt}}_k$, and $\mathcal{P}^{{\rm lse}}_k$ as 
the class $k$-uniform tight cycles, loose cycles, tight paths, and loose paths with $k\ge 3$, respectively.
Based on Theorem \ref{thm:2partition}, we have the following.

\begin{cor} \label{cor}
$\chi_{{\rm co}}(\mathcal{C}^{{\rm tyt}}_k)=\chi_{{\rm co}}(\mathcal{C}^{{\rm lse}}_k)=\chi_{{\rm co}}(\mathcal{P}^{{\rm tyt}}_k)=\chi_{{\rm co}}(\mathcal{P}^{{\rm lse}}_k)=2$.
\end{cor}

\begin{proof}
According to the definition of a tight cycle or path, its vertices can be arranged into a sequence such that each of its hyperedges consists of a consecutive subset of vertices in this sequence.
Similarly, for a loose cycle or path, although not necessarily adhering to the strict consecutive property as in tight structures, we can still order its vertices into a sequence where each hyperedge encompasses a consecutive subset of these vertices. 
Therefore, Theorem \ref{thm:2partition} is applicable to conclude the result. 
\end{proof}

\section{Multipartite hypergraphs} \label{sec:3}

\begin{lem}[Symmetric Lov\'asz Local Lemma, \cite{AS}] \label{lll}
    Let $A_1, \ldots, A_n$ be events such that $\mathbb{P}[A_i] \leq p$ for all $i$ and each $A_i$ is independent of all but at most $d$ of the other $A_j$. If $ep(d + 1) \leq 1$, then  
\[  
\mathbb{P}\left( \bigcap_{i=1}^{n} \overline{A_i }\right) > 0.  
\] 
\end{lem}

\begin{thm} \label{bound}
    For the class $\mathcal{B}_k$ of $k$-partite $k$-uniform hypergraphs, and for $d\geq 2$, we have
    \[\frac{\log_k d}{k-1}\le m_{\mathcal{B}_k}(d)\le k(1+\text{o}(1))\left(\frac{(k-1)d}{\ln d}\right)^{\frac{1}{k-1}}.\]
\end{thm}

\begin{proof}
To establish the lower bound, we first consider a family of complete $k$-partite hypergraphs. 
Given $d$, take $m=\lceil \frac{\log_k d}{k-1}\rceil$. For each $j\in [m]$, we define a hypergraph $G_j$ that consists of $k$ parts denoted by $V_j^1, V_j^2, \dots, V_j^k$, where $V_j^i$ comprises the collection of $m$-dimensional vector $v$, where each component of $v$ taking a value from the set $[k]$, and where the $j$-th component $v_j$ of $v$ is specifically equal to $i$. In simpler terms,
       \begin{align*}
           V_j^i=\{v\in \{1,2,\dots,k\}^m: v_j=i\}.
       \end{align*}
The degree of any vertex in $G_j$ is  
\begin{equation*}  
 k^{(m-1)(k-1)}\leq d,  
\end{equation*}  
since each hyperedge consists of one vertex from each of the $k$ parts $V_j^1, V_j^2, \dots, V_j^k$, and the number of ways to choose $k-1$ vectors (one from each part except the one containing the fixed vertex) is $(k^{(m-1)})^{(k-1)}$.
Suppose that $I_1, I_2, \dots, I_m$ are independent sets residing respectively in $G_1, G_2, \dots, G_m$. Since each $G_j$ is a complete $k$-partite hypergraph, it follows that for each $j$, there exists a partition set $V_j^{k_j}$ with $k_j \in \{1, 2, \dots, k\}$ such that the independent set $I_j$ does not intersect with $V_j^{k_j}$, meaning $I_j \cap V_j^{k_j} = \emptyset$. Thus, the vertex $(k_1,k_2,\dots,k_m)$ is not contained in any $I_i$. It follows that $\{I_1,I_2, \dots, I_m\}$ is not a cooperative coloring of $G_1, G_2, \dots, G_m$. Consequently, 
\[M_{\mathcal{B}_k}(d)>m\ge \frac{\log_k d}{k-1}.\]

We now proceed to establish the upper bound by demonstrating the existence of a cooperative coloring for the system $\mathcal{G} = (G_1, G_2, \dots, G_m)$ of $k$-partite $k$-uniform hypergraphs on the same vertex set $V$ and having a maximum degree of $d$, provided
$\epsilon>0$ is a small number and
the number $m$ of hypergraphs  satisfies 
\[m \geq k(1 + \epsilon)\left(\frac{(k-1)d}{\ln d}\right)^{\frac{1}{k-1}},\]
which we achieve through a semi-random construction method.

For any given vertex $v \in V$ and for each integer $i \in [k]$, we define the set $J_i(v)$ as the collection of indices $j \in [m]$ such that the vertex $v$ belongs to the $i$-th partition set $U_j^i$ of the graph $G_j$. 
This is formally defined as:  
\[ J_i(v):=\{j\in [m]:~ v\in U_j^i \}.\]  

We define $W_1$ as the set of vertices $v$ in $V$ for which the size of the set $J_1(v)$ (which contains indices $j$ such that $v$ belongs to the first partition $U_j^1$ of graph $G_j$) is at least $m/k$. Formally,
\[  
W_1 := \{ v \in V : |J_1(v)| \geq \frac{m}{k} \}.  
\]  
 For each $i$ from 2 to $k$, we define $W_i$ as the set of vertices $v$ in $V$ that do not belong to any of the previously defined sets $W_1, W_2, \dots, W_{i-1}$, and for which the size of the set $J_i(v)$ (which contains indices $j$ such that $v$ belongs to the $i$-th partition $U_j^i$ of graph $G_j$) is at least $m/k$. Formally,    
\[  
W_i := \{ v \in V \setminus (W_1 \cup W_2 \cup \cdots \cup W_{i-1}) : |J_i(v)| \geq \frac{m}{k} \}  
\]    
for $i = 2, 3, \dots, k$. 
This recursive definition ensures that each vertex $v\in V$ is assigned to exactly one of the sets $W_1, W_2, \dots, W_{k}$ 
based on the least index $i$ for which $|J_i(v)| \geq \frac{m}{k} $.

Consider the following random process:  
  
\begin{enumerate}
    \item \textbf{For each vertex in $W_i$ for $i = 1$ to $k-1$}:  
        \begin{itemize}  
            \item For each vertex $w_i \in W_i$, choose an index $j = j(w_i)$ uniformly at random from the set $J_i(w_i)$. This set $J_i(w_i)$ contains indices $j$ such that $w_i$ belongs to the $i$-th partition $U_j^i$ of graph $G_j$, which has size at least $m/k$ by the definition of $W_i$.  
            \item Assign the vertex $w_i$ to the set $I_j$, where $j$ is the randomly chosen index.  
        \end{itemize}  
  
   \item \textbf{For each vertex $w_k \in W_k$}:  
        \begin{itemize}  
            \item Construct the modified set $J'_k(w_k)$ by removing from $J_k(w_k)$ any index $j$ that satisfies both of the following conditions:  
                \begin{enumerate}  
                    \item $j = j(w_1) = j(w_2) = \dots = j(w_{k-1})$ for some $w_1, w_2, \dots, w_{k-1}$ such that each $w_i \in W_i$ for $i = 1$ to $k-1$.  
                    \item The set of vertices $\{w_1, w_2, \dots, w_{k-1}, w_k\}$ forms an edge in the graph $G_j$.  
                \end{enumerate}  
                This means we exclude indices $j$ that correspond to graphs $G_j$ where the chosen indices $j(w_1), j(w_2), \dots, j(w_{k-1})$ are all the same to $j$ and the set of vertices $\{w_1, w_2, \dots, w_{k-1}, w_k\}$ forms a hyperedge in $G_j$.  
            \item Choose an index $j \in J'_k(w_k)$ arbitrarily, as long as $J'_k(w_k)$ is non-empty. 
            \item Assign the vertex $w_k$ to the set $I_j$, where $j$ is the chosen index.  
        \end{itemize}  
\end{enumerate}

Each vertex $w_i\in W_i$ with $i\in [k-1]$ is now assigned to $I_j$ for some chosen $j$ in step 1, meaning that $w_i$ belongs to the $i$-th partition $U_j^i$ of graph $G_j$.
As a result, upon completing step 1, for each chosen $j$, $I_j$ will not contain any vertex that belongs to the $k$-th partition $U_j^k$ of graph $G_j$. 
If a vertex $w_k\in W_k$ is added to $I_j$ in step 2, then $w_k$ belongs to the $k$-th partition $U_j^k$ of graph $G_j$ by its choice.
Now we are ready to show that $I_j$ is an independent set in $G_j$ for each $j\in [m]$.

For contradiction, assume that after step 2, $I_j$ contains all the vertices $w_1, w_2, \ldots, w_k$   
forming a hyperedge in $G_j$, with each $w_i$ belonging to $U_j^i$.   
Then, each $w_i$ with $i \neq k$ was chosen for $I_j$, which   
directly implies $j = j(w_1) = j(w_2) = \dots = j(w_{k-1})$.
Therefore, $\{w_1, w_2, \dots, w_{k-1}, w_k\}$ cannot form an edge in the graph $G_j$ as otherwise $j\not\in J'_k(w_k)$, a contradiction.  

To complete the main proof, it is now sufficient to show that $J'_k(w_k)$ is nonempty for all $w_k\in W_k$ with positive probability.
For a vertex $w\in W_k$, $F_{w}$ denotes the bad event that $J'_k(w)=\emptyset$. 
For each $j \in J_k(w)$, we define the event $F^j_w$ as follows:  
  \[  
F^j_w = \left\langle \exists w_1 \in W_1, \exists w_2 \in W_2, \ldots, \exists w_{k-1} \in W_{k-1} ~\bigg| ~  
\begin{array}{l}  
\{w_1, w_2, \ldots, w_{k-1}, w\} \in E(G_j) \\  
\text{and } j(w_1) = j(w_2) = \dots = j(w_{k-1}) =j 
\end{array}  
\right\rangle.  
\]  
  In other words, $F^j_w$ is the occurrence where a set of neighbors of $w$ in $G_j$, each from a different $W_i$, share the same index value $j$ (i.e.\,they are all added to $I_j$).

For any subset $\{w_1, w_2, \ldots, w_{k-1}\}$ where each $w_i\in W_i$ is a neighbor of $w$ in the graph $G_j$, the probability that all of these neighbors share the same index value $j$ can be expressed as:  
\[  
\mathbb{P}(j(w_1) = j(w_2) = \cdots = j(w_{k-1}) = j) = \prod_{i=1}^{k-1} \frac{1}{|J_i(w_i)|}.  
\]  
This probability is bounded above by $(k/m)^{k-1}$, which arises due to the fact that, according to the definition of each $W_i$,
$|J_i(w_i)|\geq m/k$. Furthermore, 
assuming $m \geq k(1+\epsilon)(\frac{(k-1)d}{\ln d})^{\frac{1}{k-1}}$, this upper bound can be restricted to:  
  \begin{align*}
   \left(\frac{k}{m}\right)^{k-1}\leq \frac{k^{k-1}}{k^{k-1}(1+\epsilon)^{k-1} (k-1)\frac{d}{\ln d}} \leq \frac{\ln d}{(1+\epsilon)(k-1)d}.
  \end{align*} 

Given $w \in W_k$, its degree in the graph $G_j$ is denoted as $t$.   
This signifies that there exist exactly $t$ hyperedges in $G_j$ that contain the vertex $w$.   
Furthermore, we let $e_{w,i}:=\{w^i_1, w^i_2, \ldots, w^i_{k-1}, w\}$ represent the $i$-th such hyperedge, where $i\in [t]$.
For each $i\in [t]$, we define the event $E_{w,i}^j$ as follows:
  \[  
E_{w,i}^j = \left\langle j(w^i_1) = j(w^i_2) = \cdots = j(w^i_{k-1}) =j 
\right\rangle.  
\]  
In the event that $F_w^j$ does not take place, it follows that for any sequence of $w_1, w_2, \ldots, w_{k-1}$ such that $\{w_1, w_2, \ldots, w_{k-1}, w\}\subseteq E(G_j)$, it is not possible for all of these $w_i$-s (where $i\in [k-1]$) to share the same index $j$ as their respective $j(w_i)$ values. Therefore,
$\overline{F^j_w}=\bigcap_{i\in [t]}{\overline{E_{w,i}^j}}.$

Let $S=\cup_{r=1}^{i-1} \{w^r_1,w^r_2,\ldots,w^r_{k-1}\}$ and $T=S\cap \{w^i_1,w^i_2,\ldots,w^i_{k-1}\}$.
If the event $\bigcup_{1 \leq l < i} E_{w,l}^j$ occurs, 
then it is possible that $j(w^i_a)=j$ for some $w^i_a\in T$.
Let $A$ be the set of all subscripts $a$ such that $j(w^i_a)=j$ and $w^i_a\in T$.
It follows that
\[\mathbb{P}(E_{w,i}^j \mid \bigcup_{1\leq l<i}E_{w,l}^j)= \prod_{q\in [k-1]\setminus A}{\frac{1}{|J_q(w_q^i)|}}\geq  \prod_{q\in [k-1]}\frac{1}{|J_q(w_q^i)|}=\mathbb{P}(E_{w,i}^j),\]
which implies
\[\mathbb{P}(\overline{E_{w,i}^j}|\bigcap_{1\leq l<i}\overline{E_{w,l}^j)}\geq  \mathbb{P}(\overline{E_{w,i}^j}).\] 
Therefore,
\begin{align*}
    \mathbb{P}(\overline{F^j_w})&= \mathbb{P}(\bigcap_{i\in [t]}{\overline{E_{w,i}^j}})
    = \prod_{i\in [t]} \mathbb{P}(\overline{E_{w,i}^j} \mid \bigcap_{1\leq l< i}{\overline{E_{w,l}^j}})\geq \prod_{i\in [t]}\mathbb{P}(\overline{E_{w,i}^j})\\
&= \prod_{i\in [t]} \bigg(1-\mathbb{P}(j(w^i_1) = j(w^i_2) = \ldots = j(w^i_{k-1}) = j)\bigg) \\
    &\geq \bigg(1-\frac{\ln d}{(1+\epsilon)(k-1)d}\bigg)^d=\exp\bigg(d \ln\big(1-\frac{\ln d}{(1+\epsilon)(k-1)d}\big)\bigg).
\end{align*}  
Utilizing Taylor series expansion, we obtain
\begin{align*}
    d \ln\bigg(1-\frac{\ln d}{(1+\epsilon)(k-1)d}\bigg)&=-\frac{\ln d}{(1+\epsilon)(k-1)}-\frac{d}{2(1-\xi(d))^2 } \bigg(\frac{\ln d}{(1+\epsilon)(k-1)d}\bigg)^2,
\end{align*}
where \[0<\xi(d)<\frac{\ln d}{(1+\epsilon)(k-1)d}\overset{d \to \infty}{\longrightarrow} 0.\]
This statement suggests that as $d$ approaches infinity, the expression    
\[  
\frac{d}{2(1-\xi(d))^2} \left(\frac{\ln d}{(1+\epsilon)(k-1)d}\right)^2  
\]   
converges to zero. Consequently, we can establish an upper bound for this expression by a constant given by   
\[  
\frac{\epsilon(1-\epsilon)}{2(1+\epsilon)}\frac{1}{(k-1)}.  
\]  
Now,
\begin{align*}
    d \ln\bigg(1-\frac{\ln d}{(1+\epsilon)(k-1)d}\bigg)
    & \geq -\frac{\ln d}{(1+\epsilon)(k-1)}-\frac{\epsilon(1-\epsilon)}{2(1+\epsilon)}\frac{1}{(k-1)}\\
     & \geq -\frac{\ln d}{(1+\epsilon)(k-1)}-\frac{\epsilon(1-\epsilon)}{2(1+\epsilon)}\frac{\ln d}{(k-1)}\\
     & = -\frac{\ln d}{k-1} \bigg(\frac{1}{1+\epsilon}+\frac{\epsilon(1-\epsilon)}{2(1+\epsilon)}\bigg)  
     = -\frac{\ln d}{k-1}\left(1-\frac{\epsilon}{2}\right)
\end{align*}
and thus 
\begin{align*}
    1-\mathbb{P}(F_w^j)
    &\geq \exp\bigg(d \ln\big(1-\frac{\ln d}{(1+\epsilon)(k-1)d}\big)\bigg)
    \geq \exp \bigg(-\frac{\ln d}{k-1}\left(1-\frac{\epsilon}{2}\right) \bigg)
    =\left(\frac{d^{\frac{\epsilon}{2}}}{d}\right)^ {\frac{1}{k-1}}\\
    &\geq \left(\frac{4^{k-1} (\ln d)^k}{(1+\epsilon)^{k-1}(k-1)d}\right)^ {\frac{1}{k-1}}
    =4k\ln d \left( \frac{1}{k^{k-1}(1+\epsilon)^{k-1}(k-1)\frac{d}{\ln d}} \right)^ {\frac{1}{k-1}}
    \ge \frac{4k\ln d}{m}.
\end{align*}
Note that we apply above the fact that
\begin{align*}
    d^{\frac{\epsilon}{2}} \geq \text{O}((\ln d)^k)
\end{align*}
as $d$ approaches infinity.

Recall that $t$ is the degree of $w$ in  $G_j$ and $e_{w,i}:=\{w^i_1, w^i_2, \ldots, w^i_{k-1}, w\}$ is the $i$-th hyperedge incident with $w$ in $G_j$, where $i\in [t]$.
If it is not the case that $j(w^i_1)=j(w^i_2)=\cdots=j(w^i_{k-1})=j$, then we say that $e_{w,i}$ is bad.
Given that $j, j_1, j_2, \ldots, j_y$ are distinct indices drawn from $J_k(w)$, and assuming the occurrence of the conditional event $E:=F_w^{j_1} \cap F_w^{j_2} \cap \dots \cap F_w^{j_y}$, it is possible to identify a subset $Z\subseteq [t]$ with the property that for every element $i\in Z$, there exists at least one vertex in $e_{w,i}\setminus \{w\}$ that is mapped to one of the indices $j_1, j_2, \ldots, j_y$ (i.e.\,added to one of the sets $I_{j_1}, I_{j_2}, \ldots, I_{j_y}$).
We take $Z$ as large as possible. 
Now, let $p_1$ represent the probability (in the absence of any preceding conditional event) of the event $E_{[t]\setminus Z}$ that all hyperedges $e_{w,i}$ with $i \in [t] \setminus Z$ are bad. Additionally, define $p_2$ as the probability of the same event $E_{[t]\setminus Z}$ occurring, but this time conditioned on the prior occurrence of the event $E$. The choice of $Z$ implies $p_1=p_2$.
Denote $E_Z$ as the event that all hyperedges $e_{w,i}$ with $i\in Z$ are bad.
It follows
\[\mathbb{P}(\overline{F^j_w} \,|\, F_w^{j_1} \cap F_w^{j_2} \cap \dots \cap F_w^{j_y})
=\mathbb{P}(E_{[t]\setminus Z}\cap E_Z  \,|\, E)
=\mathbb{P}(E_{[t]\setminus Z}  \,|\, E)
=p_2=p_1.\]
On the other hand,
\[\mathbb{P}(\overline{F^j_w})
=\mathbb{P}(E_{[t]\setminus Z}\cap E_Z)=\mathbb{P}(E_{[t]\setminus Z})\mathbb{P}(E_{Z} \,|\, E_{[t]\setminus Z})
\leq \mathbb{P}(E_{[t]\setminus Z})=p_1.\]
Therefore,
\[\mathbb{P}(\overline{F^j_w} \,|\, F_w^{j_1} \cap F_w^{j_2} \cap \dots \cap F_w^{j_y})\geq \mathbb{P}(\overline{F^j_w})~ {\rm and~ thus~}\mathbb{P}(F^j_w \,|\, F_w^{j_1} \cap F_w^{j_2} \cap \dots \cap F_w^{j_y}) \leq \mathbb{P}(F^j_w).\]
It follows
\begin{align*}
    \mathbb{P}(F_w)=\mathbb{P}(\bigcap_{j\in J_k(w)}F^j_w)\le \prod_{j\in J_k(w)} \mathbb{P}(F^j_w)\le \left(1-\frac{4k\ln d}{m}\right)^{\frac{m}{k}}\leq \exp\left(-\frac{4k\ln d}{m}\frac{m}{k}\right)=\frac{1}{d^4}.
\end{align*} 

Let $N_{G_i}(w)$ represent the set of neighbors $u$ of $w$ in $G_i$. Furthermore, define 
  \[  
S_{G_i}(w) = N_{G_i}(w) \cup \left( \bigcup_{u \in N_{G_i}(w)} N_{G_i}(u) \right)
{\rm ~and~} 
S_{\mathcal{G}}(w):= \bigcup_{i\in [m]} S_{G_i}(w).  
\]
Clearly, $|S_{\mathcal{G}}(w)|\leq m(k-1)^2d^2$.
Now, the event $F_w$ is independent of all $F_{w'}$ but those events where $w'\in S_{\mathcal{G}}(w)$.
Since $m\leq 2d$ by Proposition \ref{proposition},
\begin{align*}
    e \cdot \mathbb{P}(F_w) \cdot (|S_{\mathcal{G}}(w)|+1) \leq e \cdot \frac{1}{d^4} \cdot  (m(k-1)^2d^2+1)
                                                          \leq  e \cdot \frac{1}{d^4} \cdot  (2(k-1)^2d^3+1)
                                                          \leq  1
\end{align*}
as $d$ approaches infinity.
With the Local Lemma (Lemma \ref{lll}) at hand, we conclude
\begin{align*}
    \mathbb{P}(\bigcap_{w\in W_k} \overline{F_w})>0,
\end{align*}
as desired. 
\end{proof}

Concluding this section, we make a note that by employing the same approach utilized in proving Proposition \ref{proposition}, it can be demonstrated that:
\begin{align*}
m_{\mathcal{B}_k}(d) \leq m_{\mathcal{B}_{k-1}}(d) \leq \cdots \leq m_{\mathcal{B}_2}(d).
\end{align*}
Recall that Aharoni \textit{et al.}\,\cite{ABCHJ} had previously established in their work \cite{ABCHJ} that   
  \[  
m_{\mathcal{B}_2}(d) \leq (1+\text{o}(1))\frac{2d}{\ln d}.  
\]  
 Thus, our main contribution lies in refining the upper bound for $m_{\mathcal{B}_k}(d)$, reducing it from $\text{O}\left(\frac{d}{\ln d}\right)$ to a tighter bound of $\text{O}\left(\frac{d}{\ln d}\right)^{\frac{1}{k-1}}$.

\section{Concluding Remarks and Open Problems} \label{sec:4}

In a hypergraph, a \textit{$k$-edge} signifies a hyperedge precisely consisting of $k$ vertices. Graphs, which solely encompass 2-edges, have traditionally been employed to model numerous real-world systems. However, by integrating 3-edges or larger edges alongside 2-edges within the hypergraph framework, the model becomes not only more compatible with these existing graph-based systems but also capable of portraying additional, more intricate interaction patterns. 
This hybrid approach of utilizing hypergraphs with both 2-edges and larger edges proves advantageous when modeling relationships in real-world systems where entities engage in group interactions of varying sizes. For instance, in the context of social networks, 2-edges aptly represent simple, pairwise friendships, whereas 3-edges or larger edges offer a means to encapsulate more sophisticated relationships such as collaborative endeavors, triadic closures (a mutually exclusive friendship circle among three or more individuals), or joint group activities. Analogously, in the realm of biological networks, 2-edges serve to depict physical interactions among proteins, whereas 3-edges or larger edges facilitate the modeling of the intricate assembly of larger protein complexes.

After completing Section \ref{sec:2}, we utilize Corollary \ref{cor} derived from Theorem \ref{thm:2partition} to determine the cooperative chromatic number for various classes of graphs, specifically $k$-uniform tight cycles, loose cycles, tight paths, or loose paths. It is noteworthy that all the hypergraphs mentioned in this context share a common characteristic: they do not contain any 2-edges. Nevertheless, Theorem \ref{thm:2partition} offers a broader perspective and has the potential to yield additional insights.

We define a chain-structured system as follows:
\begin{df}
A chain-structured system is a hypergraph such that 
 the vertices can be lined up in a circle, and each hyperedge connects a sequence of consecutive vertices along this circle.
\end{df}

Upon revisiting Theorem \ref{thm:2partition}, we discover that it permits $|A_0|=|N_0|=2$. This revelation leads us to a far more universal conclusion, stated as follows:

\begin{thm}
If $\mathcal{H}_1$ and $\mathcal{H}_2$ are two isomorphic chain-structured systems on the same vertex set such that each of them contains at most one $2$-edge (with all other hyperedges having a size of at least $3$), then they possess a cooperative coloring.
\end{thm}

An interesting open question is to determine the maximum number of 2-edges allowed in each chain-structured system while still preserving the validity of the aforementioned conclusion.
Presently, our understanding limits this count to a maximum of two.
Should there exist three 2-edges, it becomes feasible to construct two distinct paths, designated as 
$P_1$ and $P_2$, such that $V(P_1)=V(P_2)=\{1,2,3,4\}$, $E(P_1)=\{\{1,2\},\{2,3\},\{3,4\}\}$ and $E(P_2)=\{\{1,3\},\{3,2\},\{2,4\}\}$. One can easily see that $P_1$ and $ P_2$ do not admit a cooperative coloring.
Therefore, the following question is natural.

\begin{question}
    If $\mathcal{H}_1$ and $\mathcal{H}_2$ are two isomorphic chain-structured systems on the same vertex set such that each of them contains two $2$-edges (with all other hyperedges having a size of at least $3$), then do they possess a cooperative coloring?
\end{question}

In Section \ref{sec:3}, our focus was on $k$-partite $k$-uniform hypergraphs. One can see that there exists a gap between the lower and upper bounds for $m_{\mathcal{B}_k}(d)$ by Theorem \ref{bound}. Consequently, a compelling question arises: how can we close this gap?
From another perspective, the probabilistic proof of Theorem \ref{bound} heavily relies on the $k$-uniform characteristic of the $k$-partite hypergraph. To conclude, we pose the following question for future exploration.
\begin{question}
   Denote $\mathcal{B}^*_k$ as the class of $k$-partite hypergraphs. Does it hold that
    \[m_{\mathcal{B}^*_k}(d)\le \text{O}\left(\frac{d}{\ln d}\right)^{\frac{1}{k-1}}?\]
\end{question}
\noindent The result of Aharoni \textit{et al}.\,\cite{ABCHJ} gives a positive answer for this question when $k=2$.

\bibliographystyle{plain}
\bibliography{ref}

\end{document}